\documentclass[leqno, a4paper, 12pt]{article}

\usepackage{amssymb, amscd, amsthm}

\theoremstyle{plain}
\newtheorem{theorem}{Theorem}[section]
\newtheorem{lemma}[theorem]{Lemma}
\newtheorem{proposition}[theorem]{Proposition}

\theoremstyle{definition}
\newtheorem{definition}[theorem]{Definition}
\newtheorem{remark}[theorem]{Remark}
\newtheorem{remarks}[theorem]{Remarks}
\newtheorem{example}[theorem]{Example}

\newcommand\bA{{\mathbb A}}

\newcommand\bG{{\mathbb G}}

\newcommand\bP{{\mathbb P}}

\newcommand\bZ{{\mathbb Z}}

\newcommand\cC{{\mathcal C}}

\newcommand\cO{{\mathcal O}}

\newcommand\fp{\mathfrak{p}}
\newcommand\fm{\mathfrak{m}}

\newcommand\rN{{\rm N}}

\newcommand\wG{\widehat{G}}
\newcommand\wH{\widehat{H}}
\newcommand\wT{\widehat{T}}
\newcommand\wU{\widehat{U}}

\renewcommand\char{{\rm char}}

\renewcommand\deg{{\rm deg}}

\renewcommand\div{{\rm div}}
\newcommand\et{\rm{\acute{e}t}}
\newcommand\gp{{\rm gp}}
\newcommand\id{{\rm id}}

\newcommand\red{{\rm red}}

\newcommand\Aut{{\rm Aut}}
\newcommand\Br{{\rm Br}}

\newcommand\Coker{{\rm Coker}}

\newcommand\Hom{{\rm Hom}}

\newcommand\Ker{{\rm Ker}}

\newcommand\LPic{{\rm LPic}}

\newcommand\Pic{{\rm Pic}}

\newcommand\Spec{{\rm Spec}}

\newcommand\Zar{{\rm Zar}}

\title{On linearization of line bundles}

\author{Michel Brion}

\date{}

\begin{document}

\maketitle
 
\begin{abstract}
We study the linearization of line bundles and the local properties
of actions of connected linear algebraic groups, in the setting of 
seminormal varieties. We show that several classical results on
normal varieties extend to that setting, if the Zariski topology 
is replaced with the \'etale topology.
\end{abstract}

\section{Introduction}
\label{sec:int}

Linearization of line bundles in the presence of algebraic 
group actions is a basic notion of geometric invariant theory; 
it also has applications to the local properties of such actions. 
For example, given an action of a connected linear algebraic 
group $G$ on a normal variety $X$ over a field $k$, and a line 
bundle $L$ on $X$, some positive power $L^{\otimes n}$ admits a 
$G$-linearization (as shown by Mumford \cite[Cor.~1.6]{Mumford} 
when $X$ is proper, and by Sumihiro \cite[Thm.~1.6]{Sumihiro-II} 
in a more general setting of group schemes; when $k$ is 
algebraically closed of characteristic $0$, we may take for $n$ 
the order of the Picard group of $G$ as shown by Knop,
Kraft, Luna and Vust \cite[2.4]{KKLV}). It can be inferred
that $X$ is covered by $G$-stable Zariski open subsets $U_i$, 
equivariantly isomorphic to $G$-stable subvarieties of 
projectivizations of finite-dimensional $G$-modules; 
if $G$ is a split torus, then the $U_i$ may be taken affine
(see \cite[Cor.~2]{Sumihiro}, 
\cite[Thm.~3.8, Cor.~3.11]{Sumihiro-II}, 
\cite[Thm.~1.1]{KKLV}). 

The latter result does not extend to non-normal varieties, 
a classical example being the nodal curve $X$  
obtained from the projective line $\bP^1$ by identifying $0$ 
and $\infty$: the natural action of the multiplicative group 
$\bG_m$ on $\bP^1$ yields an action on $X$, and every 
$\bG_m$-stable open neighborhood of the node is the whole $X$.
Yet $X$ admits an equivariant \'etale covering by an affine
variety, namely, the union of two affine lines
meeting at their origin, where $\bG_m$ acts by scalar 
multiplication on each line.

In this article, we show that the above results on linearization
of line bundles and the local properties of algebraic group 
actions hold under weaker assumptions than normality, 
if the Zariski topology is replaced with the \'etale topology. 
For simplicity, we state our main result in the case where $k$ 
is algebraically closed:

\begin{theorem}\label{thm:main}
Let $X$ be a variety equipped with an action of a connected 
linear algebraic group $G$. 

\smallskip

\noindent
{\rm (i)} If $X$ is seminormal, then there exists a torsor 
$\pi : Y \to X$ under the character group of $G$, and a positive 
integer $n$ (depending only on $G$) such that 
$\pi^*(L^{\otimes n})$ is $G$-linearizable for any line bundle 
$L$ on $X$.

\smallskip

\noindent
{\rm (ii)} If in addition $X$ is quasi-projective, then 
it admits an equivariant \'etale covering by $G$-stable
subvarieties of projectivizations of finite-dimensional 
$G$-modules.

\smallskip

\noindent
{\rm (iii)} If $G$ is a torus and $X$ is quasi-projective, 
then $X$ admits an equivariant \'etale covering by affine 
varieties. 
\end{theorem}

We now provide details on the notions occuring in the above 
statement. By a variety, we mean a reduced separated scheme 
of finite type over the ground field 
(in particular, varieties need not be irreducible). 
Also, a line bundle $L$ on a $G$-variety $X$ is called 
$G$-linearizable, if $L$ admits a $G$-action which lifts 
the action on $X$ and is linear on fibers (see 
\cite[Sec.~1.3]{Mumford} or Subsection \ref{subsec:lin} 
for further details on linearization). 
Finally, recall the definition of seminormality: a reduced scheme 
$X$ is seminormal if every integral bijective morphism 
$f : X' \to X$ which induces an isomorphism on all residue 
fields, is an isomorphism. 
Every reduced scheme $X$ has a seminormalization map
$\sigma : X^+ \to X$, which factors the normalization map
$\eta : \tilde{X} \to X$ (see 
\cite{Andreotti-Bombieri, Greco-Traverso, Swan}).
Nodal curves are seminormal, but cuspidal curves are not.
Any projective cuspidal cubic curve $X$ has an action of 
the multiplicative group $\bG_m$ for which most line bundles 
on $X$ are not linearizable; see \cite[4.1.5]{Alexeev} or 
Example \ref{ex:cuspidal} for details. Thus, Theorem 
\ref{thm:main} (i) does not extend to arbitrarily singular 
varieties. 

That result takes a much simpler form when the character group
of $G$ is trivial, e.g., if $G$ is unipotent or semisimple. Then 
any line bundle on a seminormal $G$-variety $X$ has a 
$G$-linearizable positive power (depending only on $G$). 
As a consequence, if $X$ is quasi-projective, then it admits 
an equivariant embedding in the projectivization of a 
finite-dimensional $G$-module. Yet when the character group of
$G$ is nontrivial, the cover $\pi : Y \to X$ is infinite 
and hence $Y$ is not of finite type; but its irreducible 
components are varieties, as follows from Proposition 
\ref{prop:lift}.

In the case where $X$ is proper and $G$ is a torus, 
Theorem \ref{thm:main} (i) has been obtained by Alexeev 
(see \cite[Thm.~4.3.1]{Alexeev}) in the process of the
construction of certain moduli spaces.
His proof is based on the representability of the Picard 
functor, and hence does not extend to our general setting. 
We rather rely on results and methods from algebraic 
$K$-theory, taken from an article of Weibel (see 
\cite{Weibel}) which is chiefly concerned with the Picard 
group of Laurent polynomial rings over commutative rings. 
The connection with linearization will hopefully be clear 
from the following overview of the present article.

We work over an arbitrary field $k$; this raises some 
technical issues, as for example there exist connected 
unipotent groups having an infinite Picard group, whenever 
$k$ is imperfect (see \cite[Sec.~6.12]{KMT}).

In Section \ref{sec:prel}, we gather preliminary results
on the Picard group of linear algebraic groups, and on the
equivariant Picard group $\Pic^G(X)$ which classifies 
$G$-linearized line bundles on a $G$-scheme $X$; these results 
are variants of those in \cite{Sumihiro-II, KKV, KKLV}. In 
particular, when $X$ is reduced, we obtain an exact sequence
\[ \Pic^G(X) \stackrel{\varphi}{\longrightarrow} 
\Pic(X) \stackrel{\psi}{\longrightarrow}
\Pic(G \times X)/p_2^* \Pic(X), \]
where $\varphi$ denotes the forgetful map, and the
obstruction map $\psi$ arises from the pull-back under
the action morphism $G \times X \to X$ (see Proposition
\ref{prop:long}). Finally, we show that the obstruction 
group $\Pic(G \times X)/p_2^* \Pic(X)$ is $n$-torsion 
if $X$ is normal, where $n$ is a positive integer depending 
only on $G$ (Theorem \ref{thm:normal}).

The obstruction group is studied further in Section 
\ref{sec:semi}. We construct an injective map 
$c : H^1_{\et}(X,\wG) \to \Pic(G \times X)/p_2^* \Pic(X)$, 
where the left-hand side denotes the first \'etale cohomology 
group with coefficients in the character group of $G$ 
(viewed as an \'etale sheaf); recall that this 
cohomology group classifies $\wG$-torsors over $X$. 

Our main technical result (Theorem \ref{thm:obs}) asserts in
particular that the cokernel of $c$ is $n$-torsion for $n$ 
as above, if $X$ is a geometrically seminormal variety. When 
$G = \bG_m$, so that $\wG = \bZ$ and $G \times X = X[t,t^{-1}]$, 
the map $c : H^1_{\et}(X,\bZ) \to \Pic(X[t,t^{-1}])/\Pic(X)$ 
is a key ingredient of \cite[Sec.~7]{Weibel}, where it is
shown that $c$ is an isomorphism if $X$ is seminormal. 
For an arbitrary $G$, our proof proceeds via a reduction to 
$\bG_m$ by analyzing the behavior of the Picard group 
under various fibrations.
 
In Section \ref{sec:appl}, we present several applications 
of our analysis of the obstruction group. We first show that
linearizability is preserved under algebraic equivalence
(Proposition \ref{prop:alg}). Then we obtain a version
of Theorem \ref{thm:main} over an arbitrary base field 
(Theorems \ref{thm:lin}, \ref{thm:locqp} and \ref{thm:locaf}). 
Finally, we show that the seminormality assumption in Theorem 
\ref{thm:main} (i) and (ii) may be suppressed in prime 
characteristics (Subsection \ref{subsec:prime}).

Further applications, to the theorem of the square and the local 
properties of nonlinear group actions, will be presented in the 
follow-up article [Br14].

\medskip

\noindent
{\bf Notation and conventions.}
We consider schemes, their morphisms and their products over an 
arbitrary field $k$, with algebraic closure $\bar{k}$. All schemes 
are assumed to be separated and locally noetherian.
For any such scheme $X$, we denote by $\cO(X)$ the $k$-algebra of global 
sections of the structure sheaf, and by $\cO(X)^*$ the group of
units (i.e., invertible elements) of that algebra. The scheme obtained
from $X$ by base change via a field extension $K/k$ is denoted by $X_K$.

A smooth group scheme of finite type will be called an algebraic 
group. Throughout this article, $G$ denotes a connected algebraic
group, and $e_G$ its neutral element; a $G$-scheme is a scheme $X$ 
equipped with a $G$-action $\alpha : G \times X \to X$. We denote
by $\wG = \Hom_{\gp}(G,\bG_m)$ the character group scheme of $G$. 
We will view $\wG$ as an \'etale sheaf of free abelian groups 
of finite rank on $\Spec(k)$, and denote by $\wG(S)$ the abelian group 
of sections of $\wG$ over a scheme $S$.

\section{Preliminary results}
\label{sec:prel}

\subsection{The Picard group of a linear algebraic group}
\label{subsec:pic}

Recall that a connected unipotent group (resp.~a torus)
is said to be split, if it is an iterated extension of copies of 
the additive group $\bG_a$ (resp.~of the multiplicative group $\bG_m$). 
Also, a connected reductive group is said to be split, if it has 
a split maximal torus. We now introduce a direct generalization of 
these notions:

\begin{definition}\label{def:split}
We say that $G$ is \emph{split} if there is an exact sequence 
of algebraic groups
\begin{equation}\label{eqn:split}
1 \longrightarrow U \longrightarrow G \longrightarrow H \longrightarrow 1,
\end{equation}
where $U$ is a split connected unipotent group, and $H$ is a split
connected reductive group.
\end{definition}

\begin{remarks}\label{rem:split}
(i) The exact sequence (\ref{eqn:split}) is unique if it exists,
since $U$ is the unipotent radical of $G$. Also, the class of
split algebraic groups is stable under quotients by closed normal 
subgroups and under base change by field extensions.

\noindent
(ii) The split solvable groups in the above sense are exactly the 
extensions of split tori by split connected unipotent groups. This 
is equivalent to the usual notion of split solvable groups
(iterated extensions of copies of $\bG_a$ and $\bG_m$) in view of 
\cite[Thm.~15.4]{Borel}. 

\noindent
(iii) If $k$ is perfect, then any connected unipotent group is split
(see e.g. \cite[Cor.~15.5]{Borel}); moreover, the unipotent radical 
of $G_{\bar{k}}$ is defined over $k$. It follows that $G$ is split if 
and only if it has a split maximal torus. As a consequence, the
class of split algebraic groups is also stable under group extensions.

\noindent
(iv) If $k$ is imperfect, then there exist nontrivial forms of 
$\bG_a$, i.e., nonsplit connected unipotent groups of dimension $1$
(see e.g. \cite[Thm.~2.1]{Russell}). 

\noindent
(v) Clearly, any split group is connected and linear. Also,
$G_{\bar{k}}$ is split for any connected linear algebraic group $G$;
thus, $G_{k'}$ is split for some finite extension of fields $k'/k$. 
But such an extension may not be chosen separable; for example, 
when $G$ is a nontrivial form of $\bG_a$ (see e.g. 
\cite[Lem.~1.1, Lem.~1.2]{Russell}).
\end{remarks}

\begin{lemma}\label{lem:split}
If $G$ is split, then the sheaf $\wG$ is constant. Moreover, 
$\Pic(G)$ is finite and the natural map $\Pic(G) \to \Pic(G_K)$ 
is an isomorphism for any field extension $K/k$.
\end{lemma}

\begin{proof}
With the notation of Definition \ref{def:split}, the sheaf
$\wU$ is trivial, and hence the pull-back map $\wH \to \wG$ 
is an isomorphism. Choose a split maximal torus $T$ of $H$. 
Then the \'etale sheaf of abelian groups $\wT$ on $\Spec(k)$
is constant, and the pull-back map $\wH \to \wT$ 
is injective; it follows that $\wG$ is constant as well.

Also, $U$ (viewed as a variety) is isomorphic to an affine space 
$\bA^n$, and $G$ (viewed as a variety again) is isomorphic to
$U \times H$, since the $U$-torsor $G \to G/U \cong H$ is trivial.
It follows that the pull-back map $\Pic(H_K) \to \Pic(G_K)$ is an 
isomorphism for any field extension $K/k$. Thus, we may assume that 
$G$ is reductive. Choose a Borel subgroup $B \subset G$ and a
maximal torus $T \subset B$. By \cite[Prop.~1.9]{Iversen}, 
we have an exact sequence
\begin{equation}\label{eqn:picG} \CD
0 @>>> \wG @>{i^*}>> \wT @>{\gamma}>> \Pic(G/B) 
@>{f^*}>> \Pic(G) @>>> 0, 
\endCD \end{equation}
where $i : T \to G$ denotes the inclusion, $\gamma$ the
characteristic homomorphism, and $f : G \to G/B$ the quotient map.
Furthermore, $\Pic(G/B)$ and $\gamma$ can be explicitly described 
in terms of the root datum of $(G,T)$, in view of 
\cite[Prop.~5.2, Thm.~5.3]{Iversen}.
It follows that $\Pic(G)$ depends only on this root datum
(as already observed in \cite[Rem.~VII.1.7.a)]{Raynaud}); 
moreover, this datum is unchanged under field extensions.
\end{proof}

We also record the following observation, implicit in
\cite[Lem.~VII.1.6.1]{Raynaud}:

\begin{lemma}\label{lem:tors}
Let $X$ be a scheme, and $k'/k$ a finite extension of fields. 
Then the kernel of the pull-back map $\Pic(X) \to \Pic(X_{k'})$
is killed by $[k':k]$.
\end{lemma}

\begin{proof}
Consider a line bundle $L$ on $X$ such that $L_{k'}$ 
is trivial. Then the norm $\rN(L_{k'})$ (defined in 
\cite[II.6.5]{EGA}) is trivial as well. But 
$\rN(L_{k'}) \cong L^{\otimes [k':k]}$ by 
\cite[II.6.5.2.4]{EGA}; this yields the assertion.
\end{proof}

We now obtain a refinement of a result of Raynaud (see
\cite[Cor.~VII.1.6]{Raynaud}).
Assume that $G$ is linear, and consider a finite extension 
$k'/k$ of fields such that $G_{k'}$ is split. By Lemma 
\ref{lem:split}, the group $\Pic(G_{k'})$ is finite and 
independent of $k'$. Denote by $m = m(G)$ the exponent of 
that group (i.e., the smallest positive integer such that 
$\Pic(G_{k'})$ is $m$-torsion) and by $d = d(G)$ the greatest 
common divisor of the degrees $[k':k]$ of splitting fields 
for $G$. Finally, set $n = n(G) := d \, m$.

\begin{proposition}\label{prop:tors}
With the above notation and assumptions, the abelian group $\Pic(G_K)$ 
is killed by $n$ for any field extension $K/k$.
\end{proposition}

\begin{proof}
Choose a splitting field $k'$ and a maximal ideal $\fm$ 
of the algebra $K \otimes_k k'$. Then the quotient field 
$K' := (K \otimes_k k')/\fm$ 
is a finite extension of $K$ of degree dividing $[k':k]$. 
By Lemma \ref{lem:tors}, it follows that the kernel of the 
pull-back map $\Pic(G_K) \to \Pic(G_{K'})$ is killed by $[k':k]$. 
Moreover, $K'$ contains $k'$ and hence 
$\Pic(G_{K'})\cong \Pic(G_{k'})$ in view of Lemma \ref{lem:split}.
Thus, $\Pic(G_K)$ is killed by $[k':k] \, m$. 
\end{proof}

We say that $n$ is the \emph{stable exponent} of $\Pic(G)$. When 
$G$ is split (e.g., when $k$ is algebraically closed), $n$ is just 
the exponent of that group. For an arbitrary connected linear 
algebraic group $G$, we do not know any example where the stable 
exponent differs from the exponent of $\Pic(G)$.

\subsection{A criterion for linearizability}
\label{subsec:lin}

We first obtain a variant of another result of Raynaud
(see \cite[Cor.~VII.1.2]{Raynaud}):

\begin{lemma}\label{lem:unit}
Let $X$ be a reduced scheme. Then the multiplication map
\[ \mu: \wG(X) \times \cO(X)^* \longrightarrow \cO(G \times X)^*, 
\quad (\chi,f) \longmapsto ((g,x) \mapsto \chi(x)(g) \, f(x)) \]
is an isomorphism.
\end{lemma}

\begin{proof}
Clearly, $\mu$ is a group homomorphism. If $\mu(\chi,f) = 1$ then 
pulling back to $\{ e_G \} \times X$, we get $f = 1$ and hence 
$\chi = 1$; thus, $\mu$ is injective.
 
To show the surjectivity, consider $f \in \cO(G \times X)^*$.
Replacing $f$ with the map $(g,x) \mapsto f(g,x) \, f(e_G,x)^{-1}$,
we may assume that $f(e_G,x) = 1$ identically. Then the map 
$g \mapsto f(g,x)$ is a character for any point $x$, 
by \cite[Prop.~3]{Rosenlicht}. Therefore, $f \in \wG(X)$.
\end{proof}

\begin{lemma}\label{lem:weight}
Let $X$ be a reduced $G$-scheme. Then for any $f \in \cO(X)^*$,
there exists a unique $\chi = \chi(f) \in \wG(X)$ such that
$f(\alpha(g, x)) = \chi(x)(g) f(x)$ identically. Moreover, 
the assignement $f \mapsto \chi(f)$ yields an exact sequence
\begin{equation}\label{eqn:weight}
0 \longrightarrow \cO(X)^{*G} \longrightarrow \cO(X)^*
\stackrel{\chi}{\longrightarrow} \wG(X),
\end{equation} 
where $\cO(X)^{*G}$ denotes the subgroup of $G$-invariants in
$\cO(X)^*$.
\end{lemma}

\begin{proof}
Applying Lemma \ref{lem:unit} to 
$f \circ \alpha \in \cO(G \times X)^*$ yields $\chi \in \wG(X)$ 
and $\varphi \in \cO(X)^*$ such that 
$f(\alpha(g, x)) = \chi(x)(g) \, \varphi(x)$ identically.
By evaluating at $g= e_G$, we obtain $\varphi = f$.
This yields the first assertion; the second one follows
readily.
\end{proof}

\begin{lemma}\label{lem:res}
Let $X$ be a reduced scheme, and 
$f \in  \cO(G \times G \times X)^*$ such that 
$f(e_G,g,x) = 1 = f(g,e_G,x)$ identically. Then $f = 1$.
\end{lemma}

\begin{proof}
By Lemma \ref{lem:unit} and the isomorphism
$\widehat{G \times G} \cong \wG \times \wG$, there exist 
$\chi,\eta \in \wG(X)$ and $\varphi \in \cO(X)^*$
such that $f(g,h,x) = \chi(x)(g) \, \eta(x)(h) \, \varphi(x)$
identically. Then the assumption means that
$\eta(x)(h) \, \varphi(x) = 1 = \chi(x)(g) \, \varphi(x)$ identically, 
and hence $\varphi = 1 = \chi = \eta$.
\end{proof}

Let $X$ be a $G$-scheme, and $\pi : L \to X$ a line bundle.
Recall from \cite[Def.~1.6]{Mumford} that a $G$-linearization 
of $L$ is an action of $G$ on the scheme $L$ which lifts the given 
$G$-action $\alpha$ on $X$, and commutes with the $\bG_m$-action
on $L$ by multiplication on fibers. Equivalently, a $G$-linearization 
of $L$ is an isomorphism $\Phi : \alpha^*(L) \to p_2^*(L)$
of line bundles on $G \times X$, which satisfies the cocycle 
condition $\Phi_{gh} = \Phi_h \circ h^*(\Phi_g)$ for all points $g$, 
$h$ of $G$. 

When $X$ is reduced, this cocycle condition may be omitted, 
as shown by the following result (implicit in 
\cite[p.~577]{Sumihiro-II}; for the case where $k$ is algebraically 
closed of characteristic $0$, see \cite[Lem.~2.3]{KKLV}):

\begin{lemma}\label{lem:lin}
Let $X$ be a reduced $G$-scheme, and $L$ a line bundle on $X$. 
Then $L$ admits a $G$-linearization if and only if the line bundles
$\alpha^*(L)$ and $p_2^*(L)$ on $G \times X$ are isomorphic. 
\end{lemma}

\begin{proof}
Let $\Phi: \alpha^*(L) \to p_2^*(L)$ be an isomorphism. Since
$\alpha \circ (e_G \times \id_X) = p_2 \circ (e_G \times \id_X) = \id_X$, 
the pull-back $(e_G \times \id_X)^*(\Phi)$ is identified
with an automorphism of the line bundle $\pi : L \to X$, i.e., 
with the multiplication by some $f \in \cO(X)^*$. Replacing 
$\Phi$ with $\Phi \circ p_2^*(f)^{-1}$, we may assume that $f = 1$. 
Then $\Phi$ corresponds to a morphism 
$\beta : G \times L \to L$ such that the diagram
\[ \CD
G \times L @>{\beta}>> L \\
@V{\id_G \times \pi}VV @V{\pi}VV \\
G \times X @>{\alpha}>> X \\
\endCD \]
commutes; moreover, $\beta(e_G,z) = z$ identically. It remains
to show that $\beta$ satisfies the associativity condition
of a group action.  But the obstruction to associativity is an
automorphism of the line bundle 
$\id_{G \times G} \times \pi: 
G \times G \times L \to G \times G \times X$,
i.e., the multiplication by some 
$\varphi \in \cO(G \times G \times X)^*$. Moreover, since 
$\beta(g,\beta(e_G,z)) = \beta(g,z) = \beta(e_G,\beta(g,z))$
identically, we have $\varphi(g,e_G,x) = 1 = \varphi(e_G,g,x)$.
By Lemma \ref{lem:res}, it follows that $\varphi = 1$, i.e.,
$\beta$ is associative.
\end{proof}

\subsection{The equivariant Picard group}
\label{subsec:equiv}

The isomorphism classes of $G$-linearized line bundles
on a given $G$-scheme $X$ form an abelian group that we will
call the equivariant Picard group, and denote by $\Pic^G(X)$.
This group is equipped with a homomorphism
\[ \varphi : \Pic^G(X) \longrightarrow \Pic(X) \]
which forgets the linearization.
Also, we have a homomorphism
\[ \gamma : \wG(X) \longrightarrow \Pic^G(X) \]
which assigns to any $\chi \in \wG(X)$, the class of the trivial 
line bundle $p_1: X \times \bA^1 \to X$ on which $G$ acts by
$\beta(g,x,t) := (\alpha(g,x), \chi(x)(g) t)$.

With this notation, we may now state the following result (a version 
of \cite[Lem.~2.2, Prop.~2.3]{KKLV}):

\begin{proposition}\label{prop:long}
Let $X$ be a reduced $G$-scheme. Then there is an exact sequence
\[ 0 \to \cO(X)^{*G} \to \cO(X)^* 
\stackrel{\chi}{\longrightarrow} \wG(X) 
\stackrel{\gamma}{\longrightarrow} \Pic^G(X) 
\stackrel{\varphi}{\longrightarrow} \Pic(X) 
\stackrel{\psi}{\longrightarrow} \Pic(G \times X)/p_2^*\Pic(X), \]
where $\psi(L)$ denotes the image of $\alpha^*(L)$
in $\Pic(G \times X)/p_2^*\Pic(X)$, for any $L \in \Pic(X)$.
\end{proposition}

\begin{proof}
In view of Lemma \ref{lem:weight}, it suffices to show that
the above sequence is exact at $\wG(X)$, $\Pic^G(X)$ and $\Pic(X)$.

Exactness at $\wG(X)$: Let $\lambda \in \wG(X)$. Then 
$\gamma(\lambda) = 0$ if and only if there is an isomorphism 
of $G$-linearized line bundles 
$F: X \times \bA^1 \to X \times \bA^1$, 
where the left-hand side denotes the trivial $G$-linearized bundle,
and the right-hand side, that linearized via $\lambda$. Then $F$ is 
given by $F(x,z) = (x, f(x) z)$ for some $f \in \cO(X)^*$; moreover,
the equivariance of $F$ translates into the condition that
$f(\alpha(g, x)) = \lambda(x)(g) \, f(x)$ on $G \times X$, i.e., 
$\lambda = \chi(f)$. 

Exactness at $\Pic^G(X)$: Let $L$ be a $G$-linearized line bundle
on $X$. Then $\varphi(L) = 0$ if and only if $L$ is trivial as
a line bundle. Moreover, any $G$-linearization of the trivial
line bundle $p_1 : X \times \bA^1 \to X$ is of the form
\[ \beta(g,x,z) = (\alpha(g,x), f(g,x) z), \]
where $f \in \cO(G \times X)^*$; moreover, $f(e_G,x) = 1$
identically. By Lemma \ref{lem:unit}, it follows that 
$f(g,x) = \chi(x)(g)$ for a unique $\chi \in \wG(X)$. 

Exactness at $\Pic(X)$: Let $L$ be a line bundle on $X$. 
By Lemma \ref{lem:lin}, $L$ is $G$-linearizable if and only if 
$\alpha^*(L) \cong p_2^*(L)$. This is equivalent to the condition 
that $\alpha^*(L) \cong p_2^*(M)$ for some line bundle $M$ on $X$,
since that condition implies $L \cong M$ by pulling back to
$\{ e_G \} \times X$.
\end{proof}

\begin{remark}\label{rem:comp}
The obstruction map $\psi = \psi_{G,X}$ is compatible with 
pull-backs in the following sense: given a homomorphism 
$h : G \to G'$ of connected algebraic groups, a reduced 
$G'$-scheme $X'$ and a morphism $f : X \to X'$ of schemes 
such that $f(g \cdot x) = h(g) \cdot f(x)$ identically, 
the diagram
\[ \CD
\Pic(X') @>{\psi_{G',X'}}>> \Pic(G' \times X')/p_2^* \Pic(X') \\
@V{f^*}VV @V{(h \times f)^*}VV \\
\Pic(X) @>{\psi_{G,X}}>> \Pic(G \times X)/p_2^* \Pic(X) \\
\endCD \]  
commutes. (This follows readily from the definition of $\psi$ 
as a pull-back).
\end{remark}

Next, we consider linearization of line bundles over normal 
$G$-schemes. We will need two lemmas:

\begin{lemma}\label{lem:prod}
Let $X$ be a normal integral scheme, $\eta$ its generic point,
and $Y$ a smooth integral scheme. Then we have an exact sequence 
\[ \Pic(X) \stackrel{p_1^*}{\longrightarrow} \Pic(X \times Y) 
\stackrel{(\eta \times \id_Y)^*}{\longrightarrow} 
\Pic(\eta \times Y) = \Pic(Y_{k(X)}). \]
\end{lemma}

\begin{proof}
This can be extracted from \cite[Cor.~II.21.4.13]{EGA}, but we
prefer to give a direct proof. Note that $X \times Y$ is normal. 
Let $L$ be a line bundle on $X \times Y$ such that 
$(\eta \times \id_Y)^*(L)$ is trivial. Then 
$L \cong \cO_{X \times Y}(D)$ for some Cartier divisor $D$ on 
$X \times Y$, and there exists a rational function $f$ on 
$X \times Y$ such that $D - \div(f)$ vanishes on 
$\{ \eta \} \times Y$. Thus, $D - \div(f) = p_1^*(E)$ 
for some Weil divisor $E$ on $X$. Then $L \cong p_1^*(\cO_X(E))$; 
by descent, it follows that $E$ is a Cartier divisor.
\end{proof}

\begin{lemma}\label{lem:normal}
Let $X$ be a normal integral $G$-scheme, and $\eta$ its generic point. 
Then we have an exact sequence
\begin{equation}\label{eqn:longer}
\Pic^G(X) \stackrel{\varphi}{\longrightarrow} \Pic(X)
\stackrel{\rho}{\longrightarrow} \Pic(G_{k(X)}),
\end{equation}
where $\rho$  denotes the composition 
$\Pic(X) \stackrel{\alpha^*}{\longrightarrow}
\Pic(G \times X) \stackrel{(\id_G \times \eta)^*}{\longrightarrow}
\Pic(G_{k(X)})$.
\end{lemma}

\begin{proof}
This follows by combining Proposition \ref{prop:long} and Lemma 
\ref{lem:prod} (with $Y = G$).
\end{proof}

Lemma \ref{lem:normal} together with Proposition \ref{prop:tors} 
imply the following:

\begin{theorem}\label{thm:normal}
Let $G$ be a connected linear algebraic group, and $X$ a normal 
$G$-scheme. Then $L^{\otimes n}$ is $G$-linearizable for any line 
bundle $L$ on $X$, where $n$ denotes the stable exponent of 
$\Pic(G)$.
\end{theorem}

As mentioned in the introduction, the normality assumption in 
Theorem \ref{thm:normal} cannot be omitted in view of examples 
of nodal or cuspidal curves. We now provide details on these:

\begin{example}\label{ex:nodal}
Let $X$ be the curve obtained from $\bP^1$ by identifying 
the points $0$ and $\infty$ to the nodal point $z$. 
Denote by $\eta : \bP^1 \to X$ the normalization. Then we have 
an exact sequence
\[ 
0 \longrightarrow k^* \stackrel{\delta}{\longrightarrow}
k^* \times k^*
\longrightarrow \Pic(X)
\stackrel{\eta^*}{\longrightarrow} \Pic(\bP^1) 
\longrightarrow 0,
\]
where $\delta$ denotes the diagonal (this follows e.g. from the 
Units-Pic sequence of \cite[Prop.~7.8]{Weibel}). 
This yields an exact sequence (of abstract groups)
\[ 0 \longrightarrow \bG_m(k) \longrightarrow \Pic(X)
\stackrel{\deg}{\longrightarrow} \bZ \longrightarrow 0, \]
where the degree map is identified with $\eta^*$; the corresponding 
sequence of group schemes is exact in view of \cite[Sec.~9.2]{BLR}).

Also, the automorphism group of $X$ is isomorphic to the stabilizer
of $\{ 0 ,\infty \}$ in $\Aut(\bP^1)$; this is the semi-direct product
of $\bG_m$ (fixing $0$ and $\infty$) with the group generated by an
involution exchanging these points. In particular, the connected
automorphism group of $X$ is just $\bG_m =: G$; it acts on the Picard 
group by preserving the degree.

Let $L$ be a line bundle on $X$ of degree $n \neq 0$. Then $L$ is 
not $G$-linearizable. Indeed, any linearization of
$\eta^*(L) \cong \cO_{\bP^1}(n)$ differs from the standard 
linearization (arising from the linear action of $G$ on 
$H^0(\bP^1,\cO_{\bP^1}(1))$ with weights $0$ and $1$) by a character
of $G$, i.e., an integer $m$. Thus, the linear action of $G$ on the 
fiber $\eta^*(L)_0$ (resp. $\eta^*(L)_{\infty}$) has weight $m + n$
(resp. $m$). But $\eta^*(L)_0 \cong L_z \cong \eta^*(L)_{\infty}$
as $G$-modules, if $L$ is $G$-linearized.

By Proposition \ref{prop:long}, we have an exact sequence
\[ 0 \longrightarrow \bZ 
\stackrel{\gamma}{\longrightarrow} \Pic^G(X)
\stackrel{\varphi}\longrightarrow 
\Pic(X) \stackrel{\psi}{\longrightarrow} 
\Pic(G \times X)/p_2^*\Pic(X). \]
Using the Units-Pic exact sequence again, one may check that
$\Pic(G \times X)/p_2^* \Pic(X) \cong \bZ$ and
this identifies $\psi$ with the degree map. Thus, the $G$-linearizable 
line bundles on $X$ are exactly those of degree $0$. 
\end{example}

\begin{example}\label{ex:cuspidal}
Let $X$ be the curve obtained from $\bP^1$ by pinching the fat 
point $Z := \Spec(\cO_{\bP^1,\infty}/\fm^2)$ to the cuspidal point $z$. 
Then the normalization $\eta: \bP^1 \to X$ yields an exact sequence
\[ 
0 \longrightarrow k^* \longrightarrow (\cO_{\bP^1,\infty}/\fm^2)^*
\longrightarrow \Pic(X)
\stackrel{\eta^*}{\longrightarrow} \Pic(\bP^1) 
\longrightarrow 0.
\]
In view of the isomorphisms
$(\cO_{\bP^1,\infty}/\fm^2)^*/k^* \cong (1 + \fm)/(1 + \fm^2) 
\cong \fm/\fm^2$,
this may be identified with the exact sequence
\[ 0 \longrightarrow \bG_a(k) \longrightarrow \Pic(X)
\stackrel{\deg}{\longrightarrow} \bZ \longrightarrow 0. \]
The corresponding sequence of group schemes is exact again.

Also, the automorphism group of $X$ is isomorphic to 
$\Aut(\bP^1,\infty)$, i.e., to the automorphism group of the 
affine line; this is the semi-direct product 
$\bG_a \ltimes \bG_m$, where $\bG_m$ acts on $\bG_a$ by scalar 
multiplication. This group acts on $\Pic(X)$ by preserving the
degree, and on $\fm/\fm^2$ via the linear action of the quotient 
$\bG_m$ with weight $-1$. In particular, no nontrivial line bundle 
of degree $0$ is $\bG_m$-linearizable. But there exists a
$\bG_m$-linearizable line bundle $L$ of degree $1$: indeed, the map 
$\bP^1 \to \bP^2$, $[x:y] \mapsto [x^3 : xy^2 : y^3]$
factors through an embedding $F : X \to \bP^2$, equivariant for 
the action of $\bG_m$ on $\bP^2$ with weights $0,2,3$;
thus, we may take $L = F^*\cO_{\bP^2}(1)$. In view of Proposition
\ref{prop:long}, this yields an exact sequence
\[ 0 \longrightarrow \bZ 
\stackrel{\gamma}{\longrightarrow} \Pic^{\bG_m}(X)
\stackrel{\deg}{\longrightarrow} \bZ \longrightarrow 0. \]

We now consider the action of $G := \bG_a$ on $X$. 
In characteristic $0$, no line bundle $L$ 
of nonzero degree is $G$-linearizable. Indeed, we may assume
that $L$ has positive degree $n$, and (replacing $L$ with a 
positive power) that $L$ is very ample. 
Then one may check that the image of the pull-back map 
$H^0(X,L) \to H^0(\bP^1, \cO_{\bP^1}(n))$ 
is a hyperplane, say $H$. If $L$ is $G$-linearizable, then $H$
is stable under the representation of $G$ in 
$H^0(\bP^1, \cO_{\bP^1}(n))$, the space of homogeneous polynomials 
of degree $n$ in $x,y$ on which $G$ acts via 
$t \cdot (x,y) := (x, y + t x)$. It follows that
$H^0(\bP^1, \cO_{\bP^1}(n))$ contains a unique $G$-stable
hyperplane: the span of the monomials 
$x^n,x^{n-1}y, \ldots, x y^{n-1}$, or equivalently, the kernel 
of the evaluation map at $\infty$. But this contradicts 
the assumption that $L$ is very ample.

In contrast, if $k$ has prime characteristic $p$, then $X$ 
has a $G$-linearizable line bundle $L_p$ of degree $p$. 
Consider indeed the morphism
\[ f : \bP^1 \longrightarrow \bP^{p - 1}, \quad
[x : y] \longmapsto 
[x^p : x^{p-2} y ^2 : x^{p-3} y^3: \cdots : y^p]. \]
Then $f$ factors through a morphism $F : X \to \bP^{p-1}$,
and its image generates a $G$-stable hyperplane of
$H^0(\bP^1,\cO_{\bP^1}(p))$, viewed as the space of homogeneous 
polynomials of degree $p$ in $x,y$ with the natural action of
$G$. Thus, we may take $L_p = F^* \cO_{\bP^{p-1}}(1)$.
Note that $F$ is an embedding if $p \geq 3$; if $p = 2$, then
the map
\[ \bP^1 \longrightarrow \bP^3, \quad
[x : y] \longmapsto [x^4 : x^2 y^2 : x y ^3 : y^4] \]
yields an embedding of $X$. In any case, $X$ admits an equivariant 
embedding in the projectivization of some finite-dimensional 
$G$-module.

We now describe the obstruction to linearization, in arbitrary
characteristic. Proposition \ref{prop:long} yields an exact
sequence
\[ 0 \longrightarrow \Pic^G(X) \longrightarrow \Pic(X)
\longrightarrow \Pic(G \times X)/p_2^*\Pic(X). \]
Moreover, one may check that 
$\Pic(G \times X)/p_2^*\Pic(X) \cong \cO(G \times Z)^*/ \cO(Z)^*
\cong k[t]/k$ 
and this identifies $\psi$ with the map 
$L \mapsto \deg(L)\,t$. In particular, a line bundle 
on $X$ is $G$-linearizable if and only if its degree is a multiple 
of the characteristic.
\end{example}

\section{The obstruction to linearization on seminormal varieties}
\label{sec:semi}

\subsection{The obstruction group}
\label{subsec:obs}

Given a scheme $X$, we analyze the quotient
$\Pic(G \times X)/p_2^* \Pic(X)$. Using the section 
$e_G \times \id_X : X \to G \times X$ of $p_2$, we may identify
$\Pic(G \times X)/p_2^* \Pic(X)$ with the kernel of 
$(e_G \times \id_X)^* : \Pic(G \times X) \to \Pic(X)$. This
identifies the obstruction map $\psi$ with 
$\alpha^* - p_2^* : \Pic(X) \to \Pic(G \times X)$.

Also, recall that $\Pic(X)$ is isomorphic to the \'etale 
cohomology group $H^1_{\et}(X, \bG_m)$
(see \cite[Prop.~II.4.9]{Milne}). Thus, the Leray spectral 
sequence for $p_2$ yields an exact sequence
\begin{equation}\label{eqn:leray}
0 \longrightarrow H^1_{\et}(X,p_{2*}(\bG_m)) \longrightarrow
\Pic(G \times X) \longrightarrow H^0_{\et}(X,R^1 p_{2*}(\bG_m)). 
\end{equation}
When $X$ is reduced, Lemma \ref{lem:unit} gives an isomorphism 
of \'etale sheaves on $X$
\begin{equation}\label{eqn:prod}
\mu : \wG \times \bG_m \stackrel{\cong}{\longrightarrow} p_{2*}(\bG_m).
\end{equation} 
In particular, the map $\wG  \to p_{2*}(\bG_m)$ defines
a map $H^1_{\et}(X,\wG) \to H^1_{\et}(X,p_{2*}(\bG_m))$ and hence
a map  
\begin{equation}\label{eqn:obs} 
c = c_{G,X} : H^1_{\et}(X,\wG) \longrightarrow \Pic(G \times X).
\end{equation}

\begin{lemma}\label{lem:leray}
Let $X$ be a reduced scheme. 

\smallskip

\noindent
{\rm (i)} The above map $c$ sits in an exact sequence
\begin{equation}\label{eqn:direct}
0 \longrightarrow H^1_{\et}(X,\wG) \times \Pic(X)
\stackrel{c \times p_2^*}{\longrightarrow}
\Pic(G \times X) \longrightarrow H^0_{\et}(X,R^1 p_{2*}(\bG_m)).
\end{equation}

\smallskip

\noindent
{\rm (ii)} $c$ is compatible with pull-backs on $G$ and $X$ in 
the following sense: for any homomorphism $h : G \to G'$ 
of connected algebraic groups, and any morphism $f : X \to X'$ 
of reduced schemes, we have a commutative diagram
\[ \CD
H^1_{\et}(X',\widehat{G'}) @>{c_{G',X'}}>> \Pic(G' \times X') \\
@VVV @VVV \\
H^1_{\et}(X,\wG) @>{c_{G,X}}>> \Pic(G \times X), \\
\endCD \]
where the vertical arrows are pull-backs.

\smallskip

\noindent
{\rm (iii)} For any geometric point $\bar{x}$ of $X$, 
there is an isomorphism
\begin{equation}\label{eqn:higher}
R^1 p_{2*}(\bG_m)_{\bar{x}} \cong \Pic(G \times \Spec(\cO_{X,\bar{x}})),
\end{equation}
where $\cO_{X,\bar{x}}$ denotes the strict henselization of 
the local ring $\cO_{X,x}$; this identifies the natural map 
$\Pic(G) \to R^1 p_{2*}(\bG_m)_{\bar{x}}$ with the pull-back map 
$p_1^* : \Pic(G) \to  \Pic(G \times \Spec(\cO_{X,\bar{x}}))$.
\end{lemma}

\begin{proof}
(i) is obtained by combining (\ref{eqn:leray}) and (\ref{eqn:prod}).

(ii) By construction of $c$, it suffices to check that both
squares in the diagram
\[ \CD
H^1_{\et}(X', \widehat{G'}) @>>> H^1_{\et}(X',p'_{2*}(\bG_m)) 
@>>> \Pic(G' \times X') \\
@VVV @VVV @VVV \\
H^1_{\et}(X, \wG) @>>> H^1_{\et}(X,p_{2*}(\bG_m)) 
@>>> \Pic(G \times X) \\
\endCD \]
commute. We may view each group as a \v{C}ech cohomology group by 
\cite[Cor.~III.2.10]{Milne}.
Then the commutativity of the left square follows from the
compatibility of the map $\wG \to p_{2*}(\bG_m)$ with pull-backs,
and that of the right square is obtained by viewing
$H^1_{\et}(X,p_{2*}(\bG_m))$ as the subgroup of $\Pic(G \times X)$ 
consisting of those line bundles that are trivial on each 
$G \times U_i$ for some \'etale covering $(U_i \to X)$.

(iii) is a consequence of \cite[Thm.~III.1.15]{Milne}.
\end{proof}

When $G = \bG_m$ (so that $\wG$ is the constant sheaf $\bZ$) and 
$X$ is of finite type, the abelian group $H^1_{\et}(X, \bZ)$ is free 
of finite rank in view of \cite[Thm.~7.9]{Weibel}. Moreover, we have 
$H^0_{\et}(X,\bZ) = H^0_{\Zar}(X,\bZ) =: \bZ(X)$, and this abelian group 
is free of finite rank as well (see e.g. \cite[Ex.~7.1]{Weibel}. 
Also, $H^1_{\et}(X,\bZ) = 0$ if $X$ is normal, by 
\cite[Prop.~7.4, Thm.~7.5]{Weibel}. For a reduced scheme $X$ such 
that the normalization $\eta : \tilde{X} \to X$ is finite, the group 
$H^1_{\et}(X,\bZ)$ may be determined from $\eta$: consider indeed 
the conductor square 
\[ \CD
Y' @>>> \tilde{X} \\
@VVV  @V{\eta}VV \\
Y @>>> X \\
\endCD \]
(see e.g. \cite{Ferrand}). Then we have an exact sequence
\begin{equation}\label{eqn:lpic} 
0 \to \bZ(X) \to \bZ(\tilde{X}) \oplus \bZ(Y) \to \bZ(Y') 
\to H^1_{\et}(X,\bZ) \to 
H^1_{\et}(\tilde{X}, \bZ) \oplus H^1_{\et}(Y,\bZ) \to 
H^1_{\et}(Y',\bZ) 
\end{equation}
as follows from \cite[Thm.~7.6, Prop.~7.8]{Weibel}. For instance,
when $X$ is the nodal curve as in Example \ref{ex:nodal}, 
this yields $H^1_{\et}(X,\bZ) \cong \bZ$. 

These results extend readily to the case where $\bZ$ is replaced with 
any constant sheaf $\Lambda$ of free abelian groups of finite rank. 
Indeed, the natural map  
$H^i_{\et}(X,\bZ) \otimes_{\bZ} \Lambda \to H^i_{\et}(X,\Lambda)$ 
is then an isomorphism for all $i \geq 0$, since
$H^i_{\et}(X,\bZ^r) \cong H^i_{\et}(X,\bZ)^r$.

We now record variants of some of these results for locally constant sheaves:

\begin{lemma}\label{lem:fin}
Let $X$ be a scheme of finite type, and $\Lambda$ an \'etale 
sheaf of free abelian groups of finite rank on $\Spec(k)$. 

\smallskip

\noindent
{\rm (i)} The abelian group $\Lambda(X)$ is free of finite rank. 
Moreover, $H^1_{\et}(X,\Lambda)$ is finitely generated.

\smallskip

\noindent
{\rm (ii)} If $X$ is normal, then $H^1_{\et}(X,\Lambda)$ 
is finite.
\end{lemma}

\begin{proof}
(i) We may choose a finite Galois extension of fields $k'/k$
such that $\Lambda_{k'}$ is constant. Denoting by $\Gamma$ 
the Galois group, we have an isomorphism
$H^0(X,\Lambda) \cong H^0(X_{k'}, \Lambda_{k'})^{\Gamma}$.
Moreover, the Hochschild-Serre spectral sequence (see 
\cite[Thm.~III.2.20]{Milne}) yields an exact sequence
\[
0 \longrightarrow H^1(\Gamma, H^0_{\et}(X_{k'},\Lambda_{k'})) 
\longrightarrow H^1_{\et}(X,\Lambda) 
\longrightarrow H^1_{\et}(X_{k'},\Lambda_{k'})^{\Gamma}. 
\]
Since both abelian groups $H^0_{\et}(X_{k'},\Lambda_{k'})$ and 
$H^1_{\et}(X_{k'},\Lambda_{k'})$ are free of finite rank, this yields 
the assertions.

(ii) follows from the above exact sequence in view of the
vanishing of $H^1_{\et}(X_{k'},\Lambda_{k'})$. 
\end{proof}

Next, recall that a ring $R$ is called Nagata, or pseudo-geometric, 
if $R$ is noetherian and for any prime ideal $\fp$ of $R$ and any 
finite extension $L$ of the fraction field of $R/\fp$, the integral 
closure of $R$ in $L$ is a finite $R$-module. 
(This is equivalent to $R$ being noetherian and universally japanese,
in view of \cite[Thm.~36.5]{Nagata}; see also 
\cite[Thm.~IV.7.7.2]{EGA}). A scheme $X$ is Nagata if $X$ has an 
affine open covering by spectra of Nagata rings; this holds for example 
when $X$ is locally of finite type. 

We may now state our main technical result:

\begin{theorem}\label{thm:obs}
Let $X$ be a Nagata scheme.

\smallskip

\noindent
{\rm (i)} If $G$ is split and $X$ is connected and seminormal, 
then the map 
\[ p_1^* \times c \times p_2^*:  
\Pic(G) \times H^1_{\et}(X,\wG) \times \Pic(X) \longrightarrow 
\Pic(G \times X) \]
is an isomorphism.

\smallskip

\noindent
{\rm (ii)} If $G$ is linear and $X$ is geometrically seminormal,
then the cokernel of the map 
\[ c \times p_2^* : H^1_{\et}(X,\wG) \times \Pic(X) 
\longrightarrow \Pic(G \times X) \]
is killed by the stable exponent of $\Pic(G)$.
\end{theorem}

The proof will be given in the next two subsections. 
When $k$ has characteristic $p > 0$, the assumptions of 
(geometric) seminormality can be suppressed by tensoring with 
$\bZ[\frac{1}{p}]$, see Theorem \ref{thm:obsp}. 

When $G = \bG_a$, Theorem \ref{thm:obs} boils down 
to the isomorphism $\Pic(X \times \bA^1) \cong \Pic(X)$, 
which characterizes seminormality for reduced affine 
schemes (see \cite[Thm.~3.6]{Traverso} and \cite[Thm.~1]{Swan}).

\subsection{Some fibrations of seminormal schemes}
\label{subsec:seminormal}

From now on, we assume that all schemes under consideration are Nagata. 
This will allow us to apply results from \cite{Greco-Traverso}, 
where schemes are assumed to be locally noetherian and with 
finite normalization; also, we will use inductive arguments 
based on the Units-Pic sequence for a finite morphism of schemes, 
and the conductor square of the normalization.

We first record two easy observations:

\begin{lemma}\label{lem:des}
Let $f: X \to Y$ be a smooth surjective morphism. 
Then $X$ is seminormal if and only if $Y$ is seminormal.
\end{lemma}

\begin{proof}
By \cite[Prop.~5.1]{Greco-Traverso}, there is a cartesian square
\[ \CD
X^+ @>{f^+}>> Y^+ \\
@V{\sigma_X}VV @V{\sigma_Y}VV \\
X @>{f}>> Y,
\endCD \]
where the vertical arrows are the seminormalization maps.
So the assertion follows by descent of isomorphisms
(see \cite[Exp.~VIII, Cor.~5.4]{SGA1}).
\end{proof}

\begin{lemma}\label{lem:ac}
Let $X$ be a reduced $G$-scheme, and $\sigma : X^+ \to X$
the seminormalization. Then $X^+$ has a unique structure
of $G$-scheme such that $\sigma$ is equivariant.
\end{lemma}

\begin{proof}
The action map $\alpha : G \times X \to X$ is smooth and
surjective. Thus, we obtain a cartesian square as in the
proof of Lemma \ref{lem:des}
\[ \CD
(G \times X)^+ @>>> X^+ \\
@VVV @VVV \\
G \times X @>{\alpha}>> X,
\endCD \]
where the vertical maps are the seminormalizations.
Replacing $\alpha$ with the projection $p_2 : G \times X \to X$
and arguing similarly, we obtain an isomorphism
$(G \times X)^+ \stackrel{\cong}{\longrightarrow} G \times X^+$.
Thus, $\alpha$ lifts to a morphism 
$\alpha^+: G \times X^+ \to X^+$. Since $\sigma$ restricts to
an isomorphism on an open dense subscheme of $X$, we see that
$\alpha^+$ is unique and satisfies the properties of an algebraic
group action.
\end{proof}

Lemma \ref{lem:des} applies to any torsor under an algebraic 
group, and also to any affine bundle, i.e., a morphism $f : X \to Y$, 
where $Y$ is covered by Zariski open subsets $U_i$ such
that $f^{-1}(U_i) \cong U_i \times \bA^n$ as schemes over $U_i$. 
We now record an easy property of such bundles:

\begin{lemma}\label{lem:affine}
Let $f: X \to Y$ be an affine bundle, where $Y$ is seminormal. 
Then the pull-back maps $\cO(Y)^* \to \cO(X)^*$ and 
$\Pic(Y) \to \Pic(X)$ are isomorphisms.
\end{lemma}

\begin{proof}
Recall that $\cO(X)^* = H^0_{\Zar}(X, \cO_X^*)$, 
$\Pic(X) = H^1_{\Zar}(X,\cO_X^*)$, and likewise for $Y$. 
In view of the Leray spectral sequence for $f$ in the Zariski 
topology, it suffices to show that the natural map 
$\cO_Y^* \to f_*(\cO_X^*)$ is an isomorphism, and 
$R^1 f_*(\cO_X^*) = 0$. For this, we may assume that $f$ is 
the projection $p_1: Y \times \bA^n \to Y$. Let $U$ be an
open subscheme of $Y$, and $A := \cO(U)$. Since $A$ is reduced,
we obtain
$\cO(f^{-1}(U))^* \cong A[t_1,\ldots,t_n]^* = A^*$.
If $U$ is affine, then the pull-back map 
$\Pic(A) \to \Pic(A[t_1,\ldots,t_n])$ 
is an isomorphism, by seminormality and \cite[Thm.~3.6]{Traverso}. 
So the presheaf on $Y$ given by $U \mapsto \Pic(U \times \bA^n)$ 
has trivial stalks (since every line bundle is trivial on some 
affine neighborhood of each point). Thus, the sheaf associated
with this presheaf is trivial as well, i.e., 
$R^1 p_{1*}(\cO_{Y \times \bA^n}^*) =0$. 
\end{proof}

Next, we consider torsors under split tori; recall that any such
torsor for the fppf topology is locally trivial for the Zariski
topology.  

\begin{lemma}\label{lem:torus}
Let $f : X \to Y$ be a torsor under a split torus $T$, where
$X$ is seminormal and connected. Then there is an exact sequence
\begin{equation}\label{eqn:torus} 
0 \longrightarrow \cO(Y)^* \stackrel{f^*}{\longrightarrow} 
\cO(X)^* \stackrel{\chi}{\longrightarrow}
\wT \stackrel{\gamma}{\longrightarrow}
\Pic(Y) \stackrel{f^*}{\longrightarrow} \Pic(X)
\stackrel{\chi}{\longrightarrow} H^1_{\et}(Y,\wT),
\end{equation}
where $\gamma$ denotes the characteristic homomorphism that
assigns to any character of $T$, the class of the associated 
line bundle on $Y$.
\end{lemma}

\begin{proof}
By Lemma \ref{lem:weight}, we have a short exact sequence of 
\'etale sheaves on $Y$
\[ 0 \longrightarrow \bG_m \longrightarrow f_*(\bG_m)
\stackrel{\chi}{\longrightarrow} \wT \longrightarrow 0. \]
The associated long exact sequence of \'etale cohomology 
begins with (\ref{eqn:torus}), except that $\Pic(X)$ is
replaced with $H^1_{\et}(Y,f_*(\bG_m))$. Also, the Leray spectral 
sequence for $f$ yields an exact sequence 
\[ 0 \longrightarrow H^1_{\et}(Y,f_*(\bG_m)) \longrightarrow
\Pic(X) = H^1_{\et}(X,\bG_m) \longrightarrow
H^0_{\et}(Y,R^1f_*(\bG_m)). \]
To complete the proof, it suffices to show that 
$R^1 f_*(\bG_m) = 0$. As in the proof of Lemma \ref{lem:leray},
this is equivalent to the assertion that 
$\Pic(T \times \Spec(\cO_{Y,\bar{y}})) = 0$ 
for any geometric point $\bar{y}$ of $Y$. Since $T \cong \bG_m^n$
for some positive integer $n$, this amounts to
$\Pic(A[t_1,\ldots,t_n,t_1^{-1},\ldots,t_n^{-1}]) = 0$, where
$A$ denotes the henselian local ring $\cO_{Y,\bar{y}}$.

This vanishing follows by combining results of \cite{Weibel}.
More specifically, there is an exact sequence for any commutative
ring $R$
\[ 0 \longrightarrow \Pic(R) \longrightarrow 
\Pic(R[t]) \times \Pic(R[t^{-1}]) \longrightarrow
\Pic(R[t,t^{-1}]) \longrightarrow \LPic(R)
\longrightarrow 0, \]
where $\LPic(R) \cong H^1_{\et}(\Spec(R),\bZ)$ (see 
\cite[Lem.~1.5.1, Thm.~5.5]{Weibel}). If $R$ is seminormal,
then the maps $\Pic(R) \to \Pic(R[t])$ and 
$\Pic(R) \to \Pic(R[t^{-1}])$ are isomorphisms by 
\cite[Thm.~3.6]{Traverso} again; thus, we obtain an isomorphism
\[ \Pic(R[t,t^{-1}])/\Pic(R) \cong \LPic(R). \]
Since the (injective) map $\Pic(R) \to \Pic(R[t,t^{-1}])$ is split 
by the evaluation at $t = 1$, this yields an isomorphism
\[ \Pic(R[t,t^{-1}]) \cong \Pic(R) \oplus \LPic(R). \]
Also, $\LPic(R) \cong \LPic(R[t,t^{-1}])$ by 
\cite[Thm.~2.4]{Weibel}. Since $R[t,t^{-1}]$ is seminormal
as well, we obtain by induction
\[ \Pic(R[t_1,\ldots,t_n, t_1^{-1},\ldots,t_n^{-1}]) 
\cong \Pic(R) \oplus \bigoplus_{i=1}^n \LPic(R). \]
Moreover, $\Pic(A) = 0$ since $A$ is local, and $\LPic(A) = 0$
by \cite[Thm.~2.5]{Weibel}.
\end{proof}

\begin{remarks}
(i) The argument of Lemma \ref{lem:torus} yields a map
$\delta : H^1_{\et}(Y,\wT) \to H^2_{\et}(Y,\bG_m)$
such that the sequence 
\[
\Pic(X) \stackrel{\chi}{\longrightarrow} H^1_{\et}(Y,\wT)
\stackrel{\delta}{\longrightarrow} H^2_{\et}(Y,\bG_m)
\]
is exact; recall that $H^2_{\et}(Y,\bG_m)$ is the cohomological Brauer 
group $\Br'(Y)$. One may check that $\delta$ is the composition
\[
H^1_{\et}(Y,\wT) = H^1_{\et}(Y,\bZ) \otimes_{\bZ} \wT 
\stackrel{\id \otimes \gamma}{\longrightarrow}
H^1_{\et}(Y,\bZ) \otimes_{\bZ} H^1_{\et}(Y,\bG_m)
\stackrel{\cup}{\longrightarrow} H^2_{\et}(Y,\bG_m),
\]
where $\cup$ denotes the cup product. 

\smallskip

\noindent
(ii) For a torsor $f : X \to Y$ under an arbitrary connected linear
algebraic group $G$ (assumed in addition to be reductive if $k$ is 
imperfect), one has a longer exact sequence
\[
0 \to \cO(Y)^* \to \cO(X)^* \to \wG(k) \to \Pic(Y) \to \Pic(X)
\to \Pic(G) \to \Br'(Y) \to \Br'(X)
\]
when $Y$ is a smooth variety (see \cite[Prop.~6.10]{Sansuc}).

\smallskip

\noindent
(iii) Returning to the setting of seminormal schemes, the exact 
sequence (\ref{eqn:torus}) can be generalized to torsors under
connected algebraic groups. This will not be needed here, and 
is postponed to \cite{Brion}.
\end{remarks}

\subsection{Proof of the main result}
\label{subsec:proof}

Recall our standing assumption that all schemes are Nagata.
A key ingredient of the proof of Theorem \ref{thm:obs} is 
the following invariance property:  

\begin{proposition}\label{prop:homotopy}
Let $f : X \to Y$ be a locally trivial fibration for the \'etale 
topology, with smooth and geometrically connected fibers. Then 
$f^* : H^i_{\et}(Y,\Lambda) \to H^i_{\et}(X,\Lambda)$ 
is an isomorphism for $i = 0,1$ and any \'etale sheaf $\Lambda$
of free abelian groups of finite rank on $\Spec(k)$.
\end{proposition}

\begin{proof}
Note that the assumptions on $f$ still hold after base change by
a finite Galois extension of fields $k'/k$. Choose such an
extension so that $\Lambda_{k'}$ is trivial, and denote by $\Gamma$
the Galois group. Then the Hochschild-Serre spectral sequence 
yields a commutative diagram with exact rows 
\[ \CD
0 \to & H^1(\Gamma,H^0_{\et}(Y_{k'},\Lambda_{k'})) &
\to H^1_{\et}(Y,\Lambda) &
\to H^0(\Gamma,H^1_{\et}(Y_{k'},\Lambda_{k'})) &
\to H^2(\Gamma,H^0_{\et}(Y_{k'},\Lambda_{k'})) \\
& \downarrow & \downarrow & \downarrow 
& \downarrow \\
0 \to & H^1(\Gamma,H^0_{\et}(X_{k'},\Lambda_{k'})) &
\to H^1_{\et}(X,\Lambda) &
\to H^0(\Gamma,H^1_{\et}(X_{k'},\Lambda_{k'})) &
\to H^2(\Gamma,H^0_{\et}(X_{k'},\Lambda_{k'})), \\
\endCD \]
where the vertical arrows are pull-backs by $f$. As a consequence,
we may assume that $\Lambda$ is constant. Clearly, we may further 
assume that $\Lambda = \bZ$.

Next, using the Leray spectral sequence for $f$, it suffices to 
show that the natural map $\bZ \to f_*(\bZ)$ is an isomorphism,
and $R^1 f_*(\bZ) = 0$. In view of the local triviality assumption,
we may assume that $f$ is the projection $F \times Y \to Y$,
where $F$ is smooth and geometrically connected.
By \cite[Thm.~III.1.15]{Milne}, we are reduced to checking that 
whenever $Y = \Spec(A)$ for a henselian local ring $A$, we have
\[ \bZ(F \times Y) = \bZ \quad {\rm and} \quad
H^1_{\et}(F \times Y,\bZ) = 0. \] 
Since $Y$ is connected and $F$ is geometrically connected,
$F \times Y$ is connected by \cite[Cor.~II.4.5.8]{EGA}; this yields
the former displayed equality. The latter displayed equality is proved 
in \cite[2.5, 5.5]{Weibel} in the case where $F = \bG_m$; we argue 
along the same lines in our setting. We may assume that $Y$ is 
reduced by using \cite[Rem.~III.1.6]{Milne}. Consider the 
normalization map $\eta : \tilde{Y} \to Y$ and its conductor square
\[ \CD
S' @>>> \tilde{Y} \\
@VVV  @V{\eta}VV \\
S @>>> Y. \\
\endCD \] 
Note that $F \times \tilde{Y}$ is normal, since $F$ is smooth; 
then one easily checks that the normalization of $F \times Y$ is
$\id_F \times \eta: F \times \tilde{Y} \to F \times Y$, with conductor
$F \times S$. In other words, the conductor square of the normalization 
of $F \times Y$ is 
\[ \CD
F \times S' @>>> F \times \tilde{Y}\\
@VVV  @V{\id_F \times \eta}VV \\
F \times S @>>> F \times Y. \\
\endCD \]
So (\ref{eqn:lpic}) yields an exact sequence
\[ \bZ(F \times \tilde{Y}) \oplus \bZ(F \times S) \to
\bZ(F \times S') \to H^1_{\et}(F \times Y,\bZ) \to
H^1_{\et}(F \times \tilde{Y}, \bZ) \oplus H^1_{\et}(F \times S,\bZ). \]
Since $S$ is local henselian and strictly contained in $Y$, 
we may argue by noetherian induction and assume that 
$H^1_{\et}(F \times S,\bZ) = 0$. 
Also, since $F \times \tilde{Y}$ is normal, we have 
$H^1_{\et}(F \times \tilde{Y},\bZ) = 0$.
Thus, it suffices to show that the pull-back map 
$\bZ(F \times \tilde{Y}) \to \bZ(F \times S')$
is surjective. As $F \times Z$ is connected for any connected 
scheme $Z$, the pull-back maps
$\bZ(\tilde{Y}) \to \bZ(F \times \tilde{Y})$
and $\bZ(S') \to \bZ(F \times S')$ are isomorphisms. 
Thus, we are reduced to checking the surjectivity of the pull-back 
map $\bZ(\tilde{Y}) \to \bZ(S')$. But this is established in 
\cite{Weibel} at the end of the proof of Theorem 2.5, p.~360.
\end{proof}

\begin{remarks}\label{rem:homotopy}
(i) The above proposition applies to any $G$-torsor for the 
\'etale topology. For example, under the assumptions
of Lemma \ref{lem:torus}, this yields an isomorphism
$H^1_{\et}(Y,\wT) \cong H^1_{\et}(X,\wT)$.

\smallskip

\noindent
{\rm (ii)} The assertion of that proposition also holds
for a morphism of schemes $f: X \to Y$ which is a universal
homeomorphism, in view of
\cite[Rem.~II.3.17, Rem.~II.1.6]{Milne}. In particular, 
we obtain isomorphisms 
$H^1_{\et}(X,\Lambda) \cong H^1_{\et}(X_{\red},\Lambda) 
\cong H^1_{\et}(X^+,\Lambda)$,
where $X_{\red} \subset X$ denotes the reduced subscheme,
and $X^+$ its normalization. These isomorphisms also follow 
from \cite[Thm.~7.6, Cor.~7.6.1]{Weibel}).
\end{remarks}

Next, we obtain a local version of Theorem \ref{thm:obs}. 
To motivate its statement, recall from \cite[Thm.~44.2]{Nagata} 
that the henselian local ring $\cO_{X,\bar{x}}$ is Nagata for any 
geometric point $\bar{x}$ of a Nagata scheme $X$. By using
\cite[Prop.~5.1, Prop.~5.2]{Greco-Traverso}, it follows that
$\cO_{X,\bar{x}}$ is seminormal if so is $\cO_{X,x}$. Also, recall 
that $H^1_{\et}(\Spec(\cO_{X,\bar{x}}),\bZ) = 0$ by
\cite[Thm.~2.5]{Weibel}.

\begin{lemma}\label{lem:locpic}
Let $A$ be a henselian local ring, and $X := \Spec(A)$. 

\smallskip

\noindent
{\rm (i)} If $G$ is split and $X$ is seminormal, then the 
pull-back map $\Pic(G) \to \Pic(G \times X)$ is an isomorphism.

\smallskip

\noindent
{\rm (ii)} If $G$ is linear and $X$ is geometrically seminormal, 
then $\Pic(G \times X)$ is killed by the stable exponent 
of $\Pic(G)$.
\end{lemma}

\begin{proof}
(i) We may choose a pair $(T,B)$, where $T \subset G$ is a split 
maximal torus, and $B \subset G$ is a Borel subgroup containing 
$T$. This defines morphisms
\[ G \times X \stackrel{\varphi}{\longrightarrow} G/U \times X
\stackrel{f}{\longrightarrow} G/B \times X, \]
where $\varphi$ is a torsor under the unipotent part $U$ of $B$,
and $f$ is a torsor under $B/U \cong T$. Since $G$ is split,
so is $U$, since it is an extension of a maximal connected
unipotent subgroup of a split connected reductive group, by 
the unipotent radical of $G$ which is split as well. Thus, 
$\varphi$ is an affine bundle. Note that $G \times X$, 
$G/U \times X$ and $G/B \times X$ are seminormal by 
Lemma \ref{lem:des}. Hence $\varphi^*$ induces an isomorphism
$\Pic(G/U \times X) \cong \Pic(G \times X)$, 
in view of Lemma \ref{lem:affine}. Next, $G/U \times X$ is 
connected, since $X$ is connected and $G/U$ is geometrically 
connected. Thus, Lemma \ref{lem:torus} yields an exact sequence
\[ 
\wT \stackrel{\gamma}{\longrightarrow} \Pic(G/B \times X) 
\stackrel{f^*}{\longrightarrow} \Pic(G/U \times X)
\longrightarrow H^1_{\et}(G/B \times X,\wT). \]
Moreover, we have
$H^1_{\et}(G/B \times X,\wT) \cong 
H^1_{\et}(G/B \times X,\bZ) \otimes_{\bZ} \wT \cong
H^1_{\et}(X,\bZ) \otimes_{\bZ} \wT$
by Proposition \ref{prop:homotopy}; thus,
$H^1_{\et}(G/B \times X,\wT) = 0$ by \cite[Thm.~2.5]{Weibel}
again. Also, we have $\Pic(G/B \times X) \cong \Pic(G/B)$, 
since $X$ is local and the Picard functor of $G/B$ 
is representable by a constant group scheme. Thus, the above 
exact sequence reduces to
\[ \wT \stackrel{\gamma}{\longrightarrow} \Pic(G/B) 
\stackrel{\delta}{\longrightarrow} 
\Pic(G \times X) \longrightarrow 0, \]
where $\gamma$ is the characteristic homomorphism, and $\delta$
denotes the pull-back via the natural map $G \times X \to G/B$.
But the cokernel of $\gamma$ is isomorphic to $\Pic(G)$ in view
of the exact sequence (\ref{eqn:picG}).

(ii) By Remark \ref{rem:split} (v), we may choose a finite extension 
of fields $k'/k$ such that $G_{k'}$ is split. Then $A \otimes_k k'$ is
the product of finitely many seminormal henselian local rings. 
In view of (i), it follows that $\Pic(G \times X)_{k'})$ 
is the product of finitely many copies of $\Pic(G_{k'})$.
This yields the assertion in view of Proposition \ref{prop:tors}.
\end{proof}

\medskip

\noindent
{\sc Proof of Theorem \ref{thm:obs}.}

(i) In view of Lemmas \ref{lem:leray} and \ref{lem:locpic},
the map $\Pic(G) \to R^1p_{2*}(\bG_m)$ is an isomorphism
on stalks. Thus, $R^1p_{2*}(\bG_m)$ is the constant sheaf 
$\Pic(G)$; as a consequence, the composition of the maps 
$\Pic(G) \stackrel{p_1^*}{\longrightarrow} \Pic(G \times X) 
\longrightarrow H^0(X,R^1p_{2*}(\bG_m))$
is an isomorphism. So the assertion follows from the exact
sequence (\ref{eqn:direct}).

(ii) Let $n$ denote the stable exponent of $\Pic(G)$. Then
$\Pic(G \times \Spec(\cO_{X,\bar{x}}))$ is $n$-torsion for any
geometric point $\bar{x}$ of $X$, in view of Lemma \ref{lem:locpic}.
By Lemma \ref{lem:leray}, it follows that $R^1 p_{2*}(\bG_m)$ 
is $n$-torsion as well; hence so is the cokernel of $c \times p_2^*$,
in view of the exact sequence (\ref{eqn:direct}) again.

\section{Some applications}
\label{sec:appl}

\subsection{Linearization, algebraic equivalence and free abelian covers}
\label{subsec:alg}

We first show that linearizability of line bundles is preserved 
in an algebraic family:

\begin{proposition}\label{prop:alg}
Consider a split algebraic group $G$, a connected seminormal 
$G$-variety $X$, a smooth connected variety $Y$ and a line bundle
$L$ on $X \times Y$. Assume that the pull-back of $L$ to 
$X \times \{y \}$ is $G$-linearizable for some $y \in Y(k)$. 
Then $L$ is $G$-linearizable.
\end{proposition}

\begin{proof}
By Proposition \ref{prop:long}, we have
$\psi_{G,X}(\id_X \times y)^*(L)= 0$ in 
$\Pic(G \times X)/p_2^*\Pic(X)$; also, note that
$\psi_{G,X}(\id_X \times y)^*(L) =
(\id_{G \times X} \times y)^* \psi_{G,X \times Y}(L)$
in view of Remark \ref{rem:comp}. 

Also, we have a commutative square
by Lemma \ref{lem:leray} and Theorem \ref{thm:obs}:
\[ \CD
\Pic(G) \times H^1_{\et} (X \times Y,\wG) 
@>{(\id_X \times y)^*}>> \Pic(G) \times H^1_{\et} (X,\wG) \\
@V{p_1^* \times c_{G,X \times Y}}VV @V{p_1^* \times c_{G,X}}VV \\
\Pic(G \times X \times Y)/p_{23}^*\Pic(X \times Y) 
@>{(\id_{G \times X} \times y)^*}>> \Pic(G \times X)/p_2^* \Pic(X), \\
\endCD \]
where the vertical arrows are isomorphisms. Moreover, $Y$ is 
geometrically connected, since it is connected and has a 
$k$-rational point (see \cite[Cor.~4.5.14]{EGA}).
By Proposition \ref{prop:homotopy}, it follows that the map 
$p_1^*: H^1_{\et} (X,\wG) \to H^1_{\et} (X \times Y,\wG)$
is an isomorphism. Since $\id_X \times y$ is a section of $p_1$,
we see that 
$(\id_X \times y)^*: H^1_{\et} (X \times Y,\wG) \to H^1_{\et} (X,\wG)$
is the inverse isomorphism. As a consequence, both horizontal
arrows in the above diagram are isomorphisms. Thus,
$\psi_{G,X \times Y}(L) = 0$, i.e., $L$ is linearizable.
\end{proof}

\begin{remarks}\label{rem:alg}
(i) The seminormality assumption cannot be suppressed in the
above proposition. Consider indeed the cuspidal curve $X$ with its
action of $\bG_m$ as in Example \ref{ex:cuspidal}, and the 
line bundle $L$ on $X \times \bA^1$ associated with the isomorphism
$\bG_a \cong \Pic^0(X)$. Denote by $L_t$ the pull-back of $L$
to $X \times \{ t\}$, where $t \in k$. Then $L_0$ is linearizable,
but $L_t$ is not for $t \neq 0$.

\smallskip

\noindent
(ii) The following variant of that proposition is obtained by 
similar arguments: let $X$ be a connected, geometrically seminormal 
variety equipped with an action of a connected linear algebraic group 
$G$. Let $Y$ be a smooth connected variety equipped with a $k$-rational
point $y$, and $L$ a line bundle on $X \times Y$ which pulls back 
to a $G$-linearizable line bundle on $X \times \{ y \}$. 
Then $L^{\otimes n}$ is $G$-linearizable, where $n$ denotes the stable 
exponent of $\Pic(G)$.
\end{remarks}

Next, we obtain a lifting property for actions of connected
algebraic groups:

\begin{proposition}\label{prop:lift}
Let $X$ be a $G$-variety, $\Lambda$ an \'etale sheaf of 
free abelian groups of finite rank on $\Spec(k)$, and $\pi : Y \to X$ 
a $\Lambda$-torsor. 

\smallskip

\noindent
{\rm (i)} There is a unique action of $G$ on $Y$ such 
that $\pi$ is equivariant; moreover, this action commutes with that 
of $\Lambda$. 

\smallskip

\noindent
{\rm (ii)} The scheme $Y$ is a union of closed $G$-stable 
subvarieties $Y_i$ such that $\pi$ restricts to finite 
surjective maps $\pi_i : Y_i \to X$ for all $i$, and every $Y_i$ meets 
only finitely many $Y_j$'s. If in addition $\Lambda$ is constant,
then we may take for $Y_i$ the $\Lambda$-translates of some closed
$G$-stable subvariety $Z \subset Y$.
\end{proposition}

\begin{proof}
(i) Recall that the isomorphism classes of $\Lambda$-torsors over $X$ are 
in bijection with $H^1_{\et}(X,\Lambda)$ (see 
\cite[Prop.~III.4.6, Rem.~III.4.8]{Milne}). 

We consider first the case where $\Lambda$ is constant, and argue
as in the proof of Lemma \ref{lem:lin}.
By Proposition \ref{prop:homotopy}, the pull-back map 
$p_2^* : H^1_{\et}(X,\Lambda) \to H^1_{\et}(G \times X,\Lambda)$
is an isomorphism. It follows that the torsor $\alpha^*(\pi)$ 
(obtained from $\pi$ by pull-back via $\alpha : G \times X \to X$) 
is isomorphic to $p_2^*(\eta)$ for some $\Lambda$-torsor 
$\eta: Z \to X$. Pulling back via $e_G \times \id_X$ yields that 
$\eta \cong \pi$, i.e., $\alpha^*(\pi) \cong p_2^*(\pi)$. 
This means that there exists a morphism $\beta: G \times Y \to Y$ 
such that the diagram
\[ \CD
G \times Y @>{\beta}>> Y \\
@V{\id_G \times \pi}VV @V{\pi}VV \\
G \times X @>{\alpha}>> X \\
\endCD \]
commutes. In particular, the map $y \mapsto \beta(e_G,y)$
is an automorphism of $Y$ which lifts the identity of $X$, 
i.e., which sits in $\Aut_X(Y) \cong \Lambda(X)$. This
defines a morphism $\lambda : X \to \Lambda$ such that
$\beta(e_G, y) = \lambda(\pi(y))  y$ identically. Replacing
$\beta$ with $((- \lambda) \circ \pi) \beta$, we may assume that
$\beta(e_G,y) = y$ identically. Then the obstruction to the
associativity of $\beta$ is an automorphism of the $\Lambda$-torsor 
$\id_{G \times G} \times \pi : 
G \times G \times Y \to G \times G \times X$, 
i.e., a morphism $\varphi : G \times G \times X \to \Lambda$. 
Moreover, $\varphi(g, e_G, x) = 0 = \varphi(e_G, g, x)$ identically, 
and hence $\varphi = 0$ since $G$ is connected. 
Thus, $\beta$ is associative. 

To show that $\beta$ is unique, note that any two lifts of $\alpha$ 
differ by an automorphism of the $\Lambda$-torsor $p_2^*(\pi)$,
i.e., by a morphism $f: G \times X \to \Lambda$. Since both 
lifts pull back to the identity on $\{ e_G \} \times X$, we see 
that $f$ restricts to $0$ on $\{ e_G \} \times X$; thus, $f =0$.

To show that $\beta$ commutes with the action of 
$\Lambda$, fix $\lambda \in \Lambda$ and consider the morphism
\[ \psi : G \times Y \longrightarrow Y \times Y, \quad
(g,y) \longmapsto (y, \lambda g ( - \lambda) g^{-1} \cdot y) . \]
Clearly, the image of $\psi$ is contained in $Y \times_X Y$.
Since the latter is isomorphic to $\Lambda \times Y$ via
$(\mu,y) \mapsto (y, \mu \cdot y)$, this yields a map
$ \gamma : G \times Y \to \Lambda$ such that
$ g \lambda g^{-1} ( - \lambda) \cdot y = \gamma(g,y) \cdot y$
identically. It follows that $\gamma(e_G,y) = 0$
identically, and hence that $\gamma = 0$. 

Next, we consider an arbitrary $\Lambda$. We may then choose 
a finite Galois extension of fields $k'/k$ such that 
$\Lambda_{k'}$ is constant. Form the cartesian square
\[ \CD
Y_{k'} @>{\psi}>> Y \\
@V{\pi'}VV @V{\pi}VV \\
X_{k'} @>{\varphi}>> X. \\
\endCD \]
Then the $G$-action on $X$ lifts canonically to an action of 
$G_{k'}$ on $X_{k'}$; in turn, that action lifts uniquely to
an action of $G_{k'}$ on $Y_{k'}$ by the first step. The latter
action is equivariant under the Galois group of $k'/k$ by 
uniqueness, and hence descends to the desired action of $G$ on $Y$.

(ii) Again, we consider first the case where $\Lambda$ is 
constant. Denote by $\eta : \tilde{X} \to X$ the normalization
and form the cartesian square
\[ \CD
\tilde{Y} @>{\tau}>> Y \\
@V{\tilde{\pi}}VV @V{\pi}VV \\
\tilde{X} @>{\eta}>> X.   
\endCD \]
The $\Lambda$-torsor $\tilde{\pi}$ is trivial by Lemma 
\ref{lem:fin}; in view of Proposition \ref{prop:lift},
it follows that $\tilde{\pi}$ is equivariantly trivial
for the action of $\Lambda \times G$. In other words, we may 
choose a closed $G$-stable subvariety 
$\tilde{Z} \subset \tilde{Y}$ such that $\tilde{\pi}$ restricts
to an isomorphism $\tilde{Z}  \to \tilde{X}$, and the natural map
$\Lambda \times \tilde{Z} \to \tilde{Y}$ is an isomorphism. So 
$Y = \bigcup_{\lambda \in \Lambda} \; \lambda \cdot \tau(\tilde{Z})$;
moreover, $\tau(\tilde{Z})$ is closed in $Y$ (since 
$\tau$ is finite), of finite type, and $G$-stable. As 
the restriction $\eta \circ \tilde{\pi} : \tilde{Z} \to X$
is finite and surjective, so is $\pi : \tau(\tilde{Z}) \to X$.
Finally, since $\tau$ is finite, 
$\tau^{-1} \tau(\tilde{Z})$ meets only finitely many 
$\Lambda$-translates of $\tilde{Z}$. Equivalently, $\tau(\tilde{Z})$ 
meets only finitely many translates $\lambda \cdot \tau(\tilde{Z})$, 
where $\lambda \in \Lambda$. Thus, $Z := \tau(\tilde{Z})$ satisfies 
the assertion.

We now handle the general case, where $\Lambda$ is not necessarily
constant. Choose a finite Galois extension of fields $k'/k$
such that $\Lambda_{k'}$ is constant. By the above step, 
$Y_{k'}$ is a locally finite union of closed $G_{k'}$-stable  
subvarieties $Y'_i$, finite and surjective over $X_{k'}$; 
we may further assume that the $Y'_i$ are stable under
the Galois group. Then we may take for $Y_i$ the image of $Y'_i$
under the natural map $Y_{k'} \to Y$.
\end{proof}

We may now show that the obstruction to linearizability vanishes
by passing to a suitable free abelian cover:

\begin{theorem}\label{thm:lin}
Let $X$ be a $G$-variety. Assume either that $G$ is split and 
$X$ is seminormal, or that $G$ is linear and $X$ is geometrically 
seminormal. Then any line bundle $L$ on $X$ defines a $\wG$-torsor 
$\pi : Y \to X$ such that $\pi^*(L^{\otimes n})$ is 
$G$-linearizable, where $n$ denotes the stable exponent of $\Pic(G)$.
\end{theorem}

\begin{proof}
Consider the obstruction class 
$\psi(L^{\otimes n}) \in \Pic(G \times X)/p_2^* \Pic(X)$. 
By Theorem \ref{thm:obs}, there exists a $\wG$-torsor 
$\pi : Y \to X$ such that $\psi(L^{\otimes n})$ is the image 
of $c([\pi]) \in \Pic(G \times X)$, where 
$[\pi] \in H^1_{\et}(X,\wG)$ denotes the isomorphism class of 
$\pi$. Also, the $G$-action on $X$ lifts to an action on $Y$
in view of Proposition \ref{prop:lift}. Since the pull-back
torsor of $\pi$ under itself is trivial, and $c$ is
compatible with pull-backs by Lemma \ref{lem:leray}, 
we see that $(\id_G \times \pi)^*(\psi(L^{\otimes n})) = 0$ 
in $\Pic(G \times Y)/p_2^*\Pic(Y)$. As $\psi$ is compatible 
with pull-backs as well, this means that 
$\psi(\pi^*(L^{\otimes n})) =0$. So Proposition 
\ref{prop:long} yields the desired statement.
\end{proof}

\begin{remarks}\label{rem:univ}
(i) Given a variety $X$, there exists a universal torsor 
$f : \tilde{X} \to X$ among torsors over $X$ with
structure group a free abelian group $\Lambda$ of finite 
rank. Indeed, these torsors are classified by
the free abelian group of finite rank
\[ H^1_{\et}(X, \Lambda) \cong 
H^1_{\et}(X, \bZ) \otimes_{\bZ} \Lambda \cong
\Hom_{\bZ}(H^1_{\et}(X,\bZ)^{\vee}, \Lambda), \]
where $H^1_{\et}(X, \bZ)^{\vee} := 
\Hom_{\bZ}(H^1_{\et}(X,\bZ), \bZ)$. Thus, the assertion
follows from Yoneda's lemma. (When $k$ is algebraically closed, 
$H^1_{\et}(X,\bZ)^{\vee}$ may be viewed as the largest free abelian 
quotient of the ``enlarged fundamental group'' of 
\cite[Exp.~X, \S 6]{SGA3}.)

If $X$ is a seminormal $G$-variety, where $G$ is split, then 
$f^*(L^{\otimes n})$ admits a $G$-linearization for any line 
bundle $L$ on $X$, as follows from Theorem \ref{thm:lin}.

\smallskip

\noindent
(ii) In characteristic $0$, the assumption of (geometric) 
seminormality cannot be omitted in Theorem \ref{thm:lin}, 
as shown by the example of the cuspidal curve $X$ with
its action of $\bG_a$ as in Example \ref{ex:cuspidal}.
But in characteristic $p > 0$, that theorem may be extended
to varieties with arbitrary singularities, as we shall see
in Subsection \ref{subsec:prime}.
\end{remarks}

\subsection{Local properties of linear algebraic group actions}
\label{subsec:lpaga}

\begin{definition}\label{def:qp}
Consider a $G$-variety $X$. 

We say that $X$ is \emph{$G$-quasiprojective} if it admits
a $G$-equivariant (locally closed) immersion into the 
projectivization of a finite-dimensional $G$-module.
Equivalently, $X$ admits an ample $G$-linearized line bundle. 

We say that $X$ is \emph{locally $G$-quasiprojective} 
(for the \'etale topology) if it admits an \'etale open 
covering $(f_i : U_i \to X)_{i \in I}$, where the $U_i$ are 
$G$-quasiprojective varieties, and the $f_i$ are $G$-equivariant.
\end{definition}

Note that if $X$ is a locally $G$-quasiprojective variety on which 
$G$ acts effectively, then $G$ must be linear, since it acts 
effectively on some projective space.

\begin{theorem}\label{thm:locqp}
Let $G$ be a connected linear algebraic group, and $X$ 
a quasi-projective, geometrically seminormal $G$-variety. 
Then $X$ is locally $G$-quasiprojective.
\end{theorem}

\begin{proof}
Choose an ample line bundle $L$ on $X$. By Theorem \ref{thm:lin}, 
we may assume (after possibly replacing $L$ with a positive power) 
that $\pi^*(L)$ is $G$-linearizable for some $\wG$-torsor 
$\pi : Y \to X$. Let $(Y_i)$ be a collection of closed $G$-stable
subvarieties of $Y$ satisfying the finiteness assertions of
Proposition \ref{prop:lift} (ii). For any $y \in Y$, 
denote by $Y_y$ the (finite) union of those $Y_i$'s that contain
$y$. Then $Y_y$ is a closed $G$-stable subvariety of $Y$, finite 
over $X$; as a consequence, $L$ pulls back to an ample 
$G$-linearizable line bundle on $Y_y$. Moreover, $Y_y$ contains
an open neighborhood of $y$ in $Y$. Thus, the restriction
$\pi_y : Y_y \to X$ is \'etale at $y$, and hence over a $G$-stable
neighborhood of $y$. This yields the desired covering.
\end{proof}

Note again that the geometric seminormality assumption cannot 
be omitted in the above theorem when $\char(k) = 0$; yet the
statement will be extended to all varieties in prime 
characteristics. For split tori, we now obtain a sharper result:

\begin{theorem}\label{thm:locaf}
Let $G$ be a split torus. Then every quasiprojective $G$-variety
$X$ admits a finite \'etale $G$-equivariant cover $f : X' \to X$,
where $X'$ is the union of open affine $G$-stable subvarieties. 
We may take $X' = Y/n\wG$ for some $\wG$-torsor $Y \to X$ and
some positive integer $n$.
\end{theorem}

\begin{proof}
We adapt the proof of Theorem \ref{thm:locqp}. Choose $L$ and 
$\pi : Y \to X$ as in that proof. By Proposition \ref{prop:lift} (ii), 
there exists a closed $G$-subvariety $Z \subset Y$, finite and 
surjective over $X$, such that $Y$ is the union of the translates 
$\chi \cdot Z$, where $\chi \in \wG$, and $Z$ meets only finitely 
many such translates. 

For any $z \in Z$, denote by $Z_z$ the (finite) union of the 
translates $\chi \cdot Z$ which contain $z$. Then $Z_z$ is a closed
$G$-stable neighborhood of $z$ in $Y$, finite and surjective over 
$X$, and hence $G$-quasiprojective (since $L$ pulls back to 
an ample $G$-linearized line bundle on $Z_z$). Next, let $Z_z^0$ be the
complement in $Z_z$ of the (finitely many) intersections 
$Z_z \cap \chi \cdot Z$, where $z \notin \chi \cdot Z$. Then
$Z_z^0$ is an open $G$-quasiprojective neighborhood of $z$ in $Y$.
Moreover, there are only finitely many such subvarieties $Z_z^0$, 
where $z \in Z$, and each of them meets only finitely many of its
translates. Thus, we may choose a positive integer $n$ such that
each $Z_z^0$ is disjoint from its translates $\chi \cdot Z_z^0$,
where $\chi \in n \wG$. Then the quotient $Y/n\wG$ is finite and
\'etale over $X$, and the natural map $Z_z^0 \to Y/n\wG$ 
is an open immersion for any $z \in Z$. So the $G$-variety 
$Y/n\wG$ is the union of $G$-quasiprojective open subvarieties,
images of the $\chi \cdot Z_z^0$ for $\chi \in \wG$.

If $X$ is seminormal, then the statement follows,
since every $G$-quasiprojective variety is a union of 
open affine $G$-stable subvarieties (see e.g. the proof of 
\cite[Cor.~2]{Sumihiro}). 

In the general case, we consider the seminormalization
$\sigma : X^+ \to X$. By Lemma \ref{lem:ac}, the action of $G$ 
on $X$ lifts uniquely to an action on $X^+$ such that $\sigma$ 
is equivariant. Also, the pull-back $\sigma^*(L)$ is ample.
Thus, there exists a $\wG$-torsor $\pi^+ : Y^+ \to X^+$ 
such that $(\pi^{+})^*(L)$ is $G$-linearizable. Since the
pull-back map $H^1_{\et}(X,\wG) \to H^1_{\et}(X^+, \wG)$
is an isomorphism (see Remark \ref{rem:homotopy} (ii)), 
we have a cartesian square
\[ \CD
Y^+ @>{\tau}>> Y \\
@V{\pi^+}VV @V{\pi}VV \\
X^+ @>{\sigma}>> X,   
\endCD \]
where $\pi$ is a $\wG$-torsor, and $\tau$ the 
seminormalization. Moreover, $G$ acts on $Y$ and $Y^+$ 
so that $\tau$, $\pi$ and $\pi^+$ are equivariant 
(as follows from Proposition \ref{prop:lift}). 

Choose a positive integer $n$ as in the first step of the proof.
Then we have a cartesian square
\[ \CD
Y^+/n\wG @>{\tau_n}>> Y/n\wG \\
@V{\pi^+_n}VV @V{\pi_n}VV \\
X^+ @>{\sigma}>> X,   
\endCD \]
where $\pi_n^+ $ and $\pi_n$ are finite \'etale, and $Y^+/n\wG$ 
is the union of affine $G$-stable open subvarieties $U_i^+$. 
Since $\tau_n$ is a universal homeomorphism, the image 
$\tau_n(U_i^+) =: U_i$ is open in $Y$ for any $i$, and satisfies
$\tau_n^{-1}(U_i) = U_i^+$. Thus, $U_i$ is affine in view of a theorem 
of Chevalley (see \cite[Thm.~II.6.7.1]{EGA}). So $\pi_n$ is the 
desired morphism.
\end{proof}

\begin{remarks}\label{rem: }
(i) We do not know whether the quasi-projectivity assumption is 
necessary in Theorems \ref{thm:locqp} and \ref{thm:locaf}. 
This assumption can be omitted for a normal $G$-variety $X$, 
since every point admits a Zariski open quasi-projective 
$G$-stable neighborhood (see \cite[Lem.~8]{Sumihiro}, 
\cite[Thm.~3.8]{Sumihiro-II}, \cite[Thm.~1.1]{KKLV}). 
The proof of that result relies on properties of divisors on 
normal varieties with an algebraic group action, which do not 
extend readily to the seminormal setting.

\smallskip

\noindent
(ii) With the notation and assumptions of Theorem \ref{thm:locaf},
the variety $X'$ is not necessarily $G$-quasiprojective. For
example, if $X$ is the nodal curve with its action of $\bG_m$ as
in Example \ref{ex:nodal}, and $f : X' \to X$ has degree $n$,
then $X'$ is a cycle of projective lines $X_1,\ldots, X_n$
as follows e.g. from \cite[Exp.~X, Ex.~6.4]{SGA3}. Moreover,
$\bG_m$ acts on $X_i$ so that the point at infinity is
identified with the origin of $X_{i+1}$ for any $i$ modulo $n$.
Thus, $X'$ admits no ample $\bG_m$-linearized line bundle $L$
(otherwise, the weight of the fiber of $L$ at each $0_i$ would be
greater than the weight at $\infty_i$, with an obvious notation;
but this is impossible, since $0_n = \infty_0$).
\end{remarks}

\subsection{The case of prime characteristic}
\label{subsec:prime}

In this subsection, we assume that $k$ has characteristic $p > 0$,
and show how to extend the results of Subsections \ref{subsec:alg} 
and \ref{subsec:lpaga} to $G$-varieties with arbitrary singularities.
We will need two preliminary results, probably known but for which 
we could not find an appropriate reference.

We say that a morphism of schemes $f : X' \to X$ is 
\emph{subintegral} if $f$ is integral, bijective, and induces 
isomorphisms on all residue fields. When $X = \Spec(A)$ and
$X' = \Spec(A')$, this corresponds to the notion of subintegral
ring extension considered in \cite{Swan}.

\begin{lemma}\label{lem:rad}
Let $A \subset A'$ be a finite subintegral extension of noetherian
rings of prime characteristic $p$. Then $A'^{p^m} \subset A$ for 
$m \gg 0$, where $A'^{p^m} := \{ x^{p^m} ~\vert~ x \in A' \}$.
\end{lemma}

\begin{proof}
Denote by $N \subset A$ the ideal consisting of nilpotent elements,
and define likewise $N' \subset A'$. Then $A/N \subset A'/N'$ is a
finite subintegral extension as well. If 
$(A'/N')^{p^m} \subset A/N$, then of course $A'^{p^m} \subset A + N'$,
and hence $A'^{p^{m'}} \subset A$ for $m' \gg 0$, since $N'$ is 
nilpotent. Thus, we may assume that $A$, $A'$ are reduced.

Consider the conductor $I \subset A$, i.e., the annihilator of
the (finite) $A$-module $A'/A$. Then $I \neq 0$, as
$(A'/A)_{\eta} = 0$ for any generic point $\eta$ of $A$. 
Since the conductor square
\[ \CD
A @>>> A' \\
@VVV @VVV \\
A/I @>>> A'/I \\
\endCD \]
is cartesian, it suffices to show that $(A'/I)^{p^m} \subset A/I$
for $m \gg 0$. Note that $A/I \subset A'/I$ is again a finite 
subintegral extension of (possibly nonreduced) noetherian rings
of characteristic $p$. Thus, we may conclude by Noetherian
induction.
\end{proof}

\begin{lemma}\label{lem:ker}
Let $f : X' \to X$ be a finite subintegral morphism of 
noetherian schemes. Then the kernel and cokernel of 
$f^* : \Pic(X) \to \Pic(X')$ are killed by a power of $p$.
\end{lemma}

\begin{proof}
Consider the reduced subscheme $X'_{\red}\subset X'$. 
By a standard d\'evissage (see e.g. \cite[\S 6]{Oort}),
the kernel and cokernel of the pull-back map 
$\Pic(X') \to \Pic(X'_{\red})$ are killed by a power of $p$.
Thus, it suffices to prove the assertion in the case where 
$X'$ is reduced. Then $f$ factors through a morphism
$X' \to X_{\red}$; by a similar reduction, we may also
assume that $X$ is reduced as well. Thus, the natural map 
$\cO_X \to f_*(\cO_{X'})$ is injective. 

Note that $f^* : \Pic(X) \to \Pic(X')$ 
is the composition of the natural maps
\[ \Pic(X) = H^1_{\Zar}(X,\cO_X^*) 
\stackrel{\varphi}{\longrightarrow} H^1_{\Zar}(X,f_*(\cO_{X'}^*))
\stackrel{\psi}{\longrightarrow} H^1_{\Zar}(X',\cO_{X'}^*) = \Pic(X'). \]
Moreover, $\psi$ is an isomorphism by the Leray spectral sequence
and the vanishing of $R^1 f_*(\cO_{X'}^*)$ 
(for the latter, see \cite[Cor.~IV.21.8.2]{EGA}). 
Thus, $\Ker(f^*) = \Ker(\varphi)$ and 
$\Coker(f^*) \cong \Coker(\varphi)$. Also, we have an exact sequence
\[ 0 \longrightarrow \cO_X^* \longrightarrow f_*(\cO_{X'}^*) 
\longrightarrow \cC \longrightarrow 0 \]
for some sheaf $\cC$ on $X$; this yields a surjective map
$H^0(X,\cC) \to \Ker(f^*)$ and an injective map 
$\Coker(f^*) \to H^1(X,\cC)$. As a consequence, it suffices to 
show that $\cC$ is killed by a power of $p$. 

Since $X$ has a Zariski open covering by finitely many affine
schemes, and $f$ is affine, we may assume that $X$, $X'$ are 
affine. Then the extension $\cO(X) \subset \cO(X')$ 
satisfies the assumptions of Lemma \ref{lem:rad}, and hence 
there exists $m > 0$ such that $\cO(X')^{p^m} \subset \cO(X)$. 
It follows that $(\cO_{X',x'})^{p^m} \subset \cO_{X,f(x')}$ 
for any $x' \in X'$. Since $f$ is bijective, we conclude that 
$\cC$ is killed by $p^m$.
\end{proof}

We may now obtain a variant of Lemma \ref{lem:locpic} without any
seminormality assumption:

\begin{lemma}\label{lem:locpicp}
Let $A$ be a henselian local ring, and $X := \Spec(A)$. 

\smallskip

\noindent
{\rm (i)} If $G$ is split, then the kernel and cokernel of the
pull-back map $\Pic(G) \to \Pic(G \times X)$ are killed by 
$p^m$ for $m \gg 0$.

\smallskip

\noindent
{\rm (ii)} If $G$ is linear, then $\Pic(G \times X)$ is killed 
by $n p^m$ for $m \gg 0$, where $n$ denotes the stable exponent 
of $\Pic(G)$.
\end{lemma}

\begin{proof}
(i) Consider the seminormalization $\sigma : X^+ \to X$;
then $\id_G \times \sigma : G \times X^+ \to G \times X$ 
is the seminormalization by \cite[Prop.~5.1]{Greco-Traverso}.
In view of Lemma \ref{lem:ker}, it follows that the kernel 
and cokernel of the pull-back map 
$\Pic(G \times X) \to \Pic(G \times X^+)$ are killed by $p^m$
for $m \gg 0$. Also, the pull-back map 
$\Pic(G) \to \Pic(G \times X^+)$ is an isomorphism by Lemma
\ref{lem:locpic}. This yields the statement.

(ii) follows from (i) as in the proof of Lemma \ref{lem:locpic}.
\end{proof}

By arguing as in the proof of Theorem \ref{thm:obs}, we see that
Lemma \ref{lem:locpicp} implies the following variant of 
that theorem:

\begin{theorem}\label{thm:obsp}
Let $X$ be a noetherian scheme.

\smallskip

\noindent
{\rm (i)} If $G$ is split, then the kernel and cokernel 
of the map  
\[ p_1^* \times c \times p_2^*:  
\Pic(G) \times H^1_{\et}(X,\wG) \times \Pic(X) \longrightarrow 
\Pic(G \times X) \]
are killed by $p^m$ for $m \gg 0$.

\smallskip

\noindent
{\rm (ii)} If $G$ is linear, then the cokernel of the map
\[ c \times p_2^* : H^1_{\et}(X,\wG) \times \Pic(X) 
\longrightarrow \Pic(G \times X) \]
is killed by $n p^m$ for $m \gg 0$, where $n$ denotes 
the stable exponent of $\Pic(G)$.
\end{theorem}

Using Theorem \ref{thm:obsp}, we may extend the results of Subsections
\ref{subsec:alg} and \ref{subsec:lpaga} to arbitrary $G$-varieties,
provided that all the line bundles under consideration are replaced 
with their $p^m$th powers for $m \gg 0$. For example, 
every quasi-projective $G$-variety is locally $G$-quasiprojective 
when $G$ is arbitrary, and is $G$-quasiprojective when $G$ has 
trivial character group.

\bigskip

\noindent
{\bf Acknowledgements}. 

I thank St\'ephane Druel, H\'el\`ene Esnault, St\'ephane Guillermou, 
Marc Levine, Tam\'as Szamuely and Chuck Weibel for very helpful 
discussions or e-mail exchanges. Thanks are also due to the referee
for valuable comments.

Most of this article was written during a stay at the Centro de 
Investigaci\'on for Matem\'atica in December 2013. I thank CIMAT 
for providing excellent working conditions, and Leticia Brambila 
and Claudia Reynoso for their invitation and kind hospitality.

\end{document}